\documentclass{amsart}
\usepackage{graphicx, amssymb}
\newtheorem{thm}{Theorem}[]
\newtheorem{cor}[thm]{Corollary}
\newtheorem{lem}[]{Lemma}
\newtheorem{prop}[]{Proposition}

\theoremstyle{definition}
\newtheorem{defn}[]{Definition}
\theoremstyle{remark}
\newtheorem{rem}[]{Remark}
\numberwithin{equation}{section}

\newcommand{\set}[1]{\left\{#1\right\}}
\newcommand{\Real}{\mathbb R}
\newcommand{\R}{\mathbb{R}}
\newcommand{\ve}{\varepsilon}

\newcommand{\seq}[1]{\left\{ #1 \right\}_{k=0}^{\infty}}

\title[]{Multiplier Sequences for Simple Sets of Polynomials}%
\author{Tam\'as Forg\'acs${}^{\dag}$}
\address{$\dag$- corresponding author \newline \indent Department of Mathematics\newline \indent
California State University, Fresno \newline \indent 5245 North Backer Ave., M/S PB 108 \newline \indent Fresno, CA 93740-8001}
\email{tforgacs@csufresno.edu}
\author{James Tipton${}^{\ddag}$}
\address{$\ddag$-Department of Mathematics\newline \indent
University of Iowa  \newline \indent 14 MacLean Hall \newline \indent Iowa City, IA 52242-1419}
\email{james-tipton@uiowa.edu}
\author{Benjamin Wright${}^{\S}$}
\address{$\S$- Department of Mathematics\newline \indent
California State University, Fresno \newline \indent 5245 North Backer Ave., M/S PB 108 \newline \indent Fresno, CA 93740-8001}
\email{wrightbw@mail.fresnostate.edu}

\thanks{Research partially supported by CURM (NSF grant DMS-063664) and the College of Science and Mathematics at California State University, Fresno.}

\begin{document}
\maketitle

\begin{abstract} In this paper we give a new characterization of simple sets of polynomials $B$ with the property that the set of $B$-multiplier sequences contains all $Q$-multiplier sequence for every simple set $Q$. We characterize sequences of real numbers which are multiplier sequences for every simple set $Q$, and obtain some results toward the partitioning of the set of classical multiplier sequences. {\bf 30C15}\\
\smallskip

\end{abstract}


\section{Introduction} In their seminal work \cite{PS}  P\'olya and Schur completely characterized all sequences of real numbers $\{\gamma_k\}_{k=0}^{\infty}$ satisfying the following property. 

\medskip

\noindent {\bf Property A.} Given any real polynomial
\[
f(x)=\sum_{k=0}^n a_k x^k
\]  
with only real zeros, the polynomial
\[
\Gamma[f(x)]=\sum_{k=0}^n a_k\gamma_k x^k
\]
also has only real zeros. 

\medskip

\noindent Such sequences are called classical multiplier sequences (of the first kind), where the word `classical' refers to the classical/standard basis for the polynomial ring $\R[x]$. Since the late $19^{th}$ century there has been quite a bit of work done in the area of multiplier sequences. Early contributions were made by C. Hermite, E.N. Laguerre, J. Jensen, G. P\'olya, J. Shur, and P. Tur\'an, while Bleecker, T. Craven, G. Csordas conducted most of their research in this area in the late $20^{th}$ and early $21^{st}$ century. For a list of papers highlighting the contributions of these mathematicians to the theory of multiplier sequences, we refer the reader to the extensive bibliography of \cite{andrzej}. 
\newline \indent A natural question arising from the study of classical multiplier sequences is the following: which sequences of real numbers $\{\gamma_k\}_{k=0}^{\infty}$ possess the analog of Property A, where we expand our polynomials in a basis different from the standard one? In \cite{andrzej} Piotrowski characterized all multiplier sequences for the Hermite and generalized Hermite (or $\mathcal{H}^{(\alpha)})$ bases with $\alpha >0$. He also gave a characterization of all bases which share multiplier sequences with the standard basis. We now recall this result.

\begin{defn} Let $Q=\seq{q_k(x)}$ be a set of polynomials. $Q$ is called a {\it simple set of polynomials}, if $\deg q_k(x)=k$ for $k=0,1,2,\ldots$. 
\end{defn}
 
\begin{defn} Let $\{\gamma_k\}_{k=0}^{\infty}$ be a sequence of real numbers, and let $Q=\seq{q_k(x)}$ be a simple set of polynomials.  If 
\[
\Gamma[f(x)]=\sum_{k=0}^n a_k\gamma_k q_k(x)
\]
has only real zeros whenever
\[
f(x)=\sum_{k=0}^n a_k q_k(x)
\]  
has only real zeros, we say that $\{\gamma_k\}_{k=0}^{\infty}$ is a $Q$-multiplier sequence.
\end{defn}
\begin{thm}
[Lemma 157 in \cite{andrzej}] \label{andlem} Let $Q = \{q_k(x)\}^{\infty}_{k=0}$ be a simple set of polynomials. Suppose in addition that $\{c_k\}_{k=0}^{\infty}$ is a sequence of non-zero real numbers, $\alpha \in \R \setminus \{0\}$, and $\beta \in \R$. Let $\widehat{Q}= \{\widehat{q}_k(x)\}_{k=0}^{\infty}$, where we define
\[
\widehat{q}_k(x)=c_k q_k(\alpha x+\beta)	\qquad \qquad (k=0,1,2,...). 
\]
Then $\{\gamma_k\}_{k=0}^{\infty}$ is a $Q$-multiplier sequence if and only if $\{\gamma_k\}_{k=0}^{\infty}$ is a $\widehat{Q}$-multiplier sequence.
\end{thm}
It follows that the only simple sets of polynomials which share multiplier sequences with the standard basis are those obtained from the standard basis by affine transformations as described in Lemma \ref{andlem}. 
\newline \indent Our paper pursues two avenues of investigation. In Section \ref{equivalencethm} we develop an alternative characterization of bases which share multiplier sequences. Instead of affine transformations (as in Lemma \ref{andlem}), our characterization uses the existence of certain geometric multiplier sequences for two given bases. In Section \ref{partitions} we investigate the question whether or not, and in what sense, the set of classical multiplier sequences can be partitioned. We conclude the paper with a list of open problems in Section \ref{concl}.

\section{An equivalence theorem}\label{equivalencethm}
We begin with two simple results.
\begin{lem}\label{powers}
Let $B = \{b_k(x)\}_{k=0}^\infty$ be any simple set of polynomials, and suppose $\{\gamma_k\}_{k=0}^\infty$ is a $B$-multiplier sequence. Then $\{\gamma_k^m\}_{k=0}^\infty$ is also a $B$-multiplier sequence for all $m\in\mathbb{N}.$
\end{lem}
\begin{proof}
Let $f(x) = \sum_{k=0}^nc_kb_k(x)$ be a polynomial with only real zeros. Since $\{\gamma_k\}_{k=0}^\infty$ is a $B$-multiplier sequence, $m$ consecutive application of this sequence yields the polynomial 
\[
p_m(x):=\sum_{k=0}^n\gamma_k^mc_kb_k(x)
\]
with only real zeros. Hence $\{\gamma_k^m\}_{k=0}^\infty$ is a $B$-multiplier sequence for all $m\in\mathbb{N}$.
\end{proof}

\begin{lem} \label{changebasis}
Let $Q = \{q_k\}_{k=0}^\infty$ and $B = \{b_k\}_{k=0}^\infty$ be simple sets of polynomials. Suppose $\alpha>1$, and suppose that $\displaystyle{\left \{\alpha^k\right \}_{k=0}^\infty}$ and $\displaystyle{\left \{\left(1/\alpha\right)^k\right\}_{k=0}^\infty}$ are both $B$-multiplier sequences. Write $q_k(x) = \sum_{j=0}^ka_{k,j}b_j(x)$, and define $\widehat{Q} := \{\widehat{q}_k\}_{k=0}^\infty$ where $\widehat{q}_k = \sum_{j=0}^k \alpha^j a_{k,j}b_j(x)$. Then every $Q$-multiplier sequence is also a $\widehat{Q}$-multiplier sequence.
\end{lem}

\begin{proof}
Let $\{\gamma_k\}_{k=0}^\infty$ be a Q-multiplier sequence and let $f(x) = \sum_{k=0}^n c_k\widehat{q}_k(x)$ be a polynomial with only real zeros. After substituting in the expressions for the $\hat{q}_k$s we obtain

\[ 
f(x) = \sum_{k=0}^n \sum_{j=0}^k c_k\alpha^ja_{k,j}b_j(x) = \sum_{j=0}^n \alpha^jb_j(x)\left(\sum_{k=j}^nc_ka_{k,j}\right). 
\]

Since $\displaystyle{\left\{\left(\frac{1}{\alpha}\right)^k\right\}_{k=0}^\infty}$ is a $B$-multiplier sequence, the function
\begin{eqnarray*}
 f^*(x) &=& \sum_{j=0}^n\left(\frac{1}{\alpha}\right)^j\alpha^jb_j(x)\left(\sum_{k=j}^nc_ka_{k,j}\right) = \sum_{j=0}^n b_j(x)\sum_{k=j}^nc_ka_{k,j} = \sum_{k=0}^n \sum_{j=0}^k c_ka_{k,j}b_j(x) \\ &=& \sum_{k=0}^nc_kq_k(x)
\end{eqnarray*}
also has only real zeros. Applying the Q-multiplier sequenc $\{\gamma_k\}_{k=0}^\infty$ to $f^*(x)$ gives
\[
\sum_{k=0}^n\gamma_kc_kq_k(x) = \sum_{k=0}^n \sum_{j=0}^k \gamma_kc_ka_{k,j}b_j(x) = \sum_{j=0}^n b_j(x)\sum_{k=j}^n\gamma_kc_ka_{k,j},
\]
which also has only real zeros. Finally, since $\displaystyle{\left \{\alpha^k\right \}_{k=0}^\infty}$ is a $B$-multiplier sequence,
\[
\sum_{j=0}^n \alpha^jb_j(x)\sum_{k=j}^n\gamma_kc_ka_{k,j} = \sum_{k=0}^n \sum_{j=0}^n \gamma_kc_k\alpha^ja_{k,j}b_j(x) = \sum_{k=0}^n\gamma_kc_k\widehat{q}_k(x)
\]
has only real zeros, establishing that $\{\gamma_k\}_{k=0}^\infty$ is a $\widehat{Q}$-multiplier sequence. The proof is complete.
\end{proof}

We are now in position to prove our main theorem.
\begin{thm} \label{classicalequivalence} Let $B\,=\,\{b_k(x)\}_{k=0}^\infty$ and $Q = \{q_k(x)\}_{k=0}^\infty$ be simple sets of polynomials. If there exists an $\alpha>1$ such that both $\{\alpha^k\}_{k=0}^{\infty}$ and $\{\left(\frac{1}{\alpha}\right)^k\}_{k=0}^\infty$ are $B$-multiplier sequences, then every $Q$-multiplier sequence is also a $B$-multiplier sequence. 
\end{thm}

\begin{rem}\label{allclassical}  Before we prove the theorem, we point out that in the statement of the theorem, the simple set $Q$ is completely arbitrary. Thus if an $\alpha>1$ as in the statement of the theorem exists, then {\it every $Q$-multiplier sequence for every simple set $Q$} is a $B$-multiplier sequence.
\end{rem}
\begin{proof} [Proof (of Theorem \ref{classicalequivalence})]

\noindent For each $k=0,1,2,3,...$ we can choose $a_{k,j} \in \Real$ such that
\begin{eqnarray*}
q_k(x)\,&=&\,\sum_{j=0}^ka_{k,j}b_j(x).
\end{eqnarray*}
Suppose that $\alpha>1$ is as in the statement of the theorem, and define 
\[
\alpha_{\ell} := \alpha^{\ell} \qquad \forall\ell \in \mathbb{N}.
\]
From Lemma \ref{powers} we know that $\{\alpha_{\ell}^k\}_{k=0}^\infty$ is a $B$-multiplier sequence for all $\ell\in\mathbb{N}$. Suppose $\{\gamma_k\}_{k=0}^\infty$ is a $Q$-multiplier sequence. By Lemma \ref{changebasis} $\{\gamma_k\}_{k=0}^\infty$ is also a multiplier sequence for each of the sets
\[
Q_{\alpha_{\ell}} = \left \{q_k^{\alpha_{\ell}}(x)\right\}_{k=0}^\infty = \left\{\sum_{j=0}^k \alpha_{\ell}^j a_{k,j}b_j(x)\right\}_{k=0}^\infty.
\]
By Lemma \ref{andlem}, $\{\gamma\}_{k=0}^\infty$ is then also a multiplier sequence for each of the sets
\[
Q_{\alpha_{\ell}}^*:=\left \{p_k^{\alpha_{\ell}}\right\}_{k=0}^\infty=\left\{\dfrac{1}{\alpha_{\ell}^k a_{k,k}}q_k^{\alpha_{\ell}}(x)\right\}_{k=0}^\infty=\left \{\sum_{j=0}^k\dfrac{a_{k,j}b_j(x)}{a_{k,k}\alpha_{\ell}^{k-j}}\right \}_{k=0}^\infty.
\]

\noindent Suppose now that $f(x)\,=\,\sum_{k=0}^nm_kb_k(x)$ has only real zeros.  For each $\ell\in\mathbb{N}$ we can expand $f$ in the basis $Q_{\alpha_{\ell}}^*$:\\
\begin{displaymath}
f(x)\,=\,\sum_{k=0}^nc_{{\alpha_{\ell}},k}p_k^{\alpha_{\ell}}(x).
\end{displaymath}
Since $\{\gamma_k\}_{k=0}^{\infty}$ is a $Q_{\alpha_{\ell}}^*$-multiplier sequence for every $\ell$, $f_{\alpha_{\ell}}(x)\,:=\,\sum_{k=0}^nc_{{\alpha_{\ell}},k}\gamma_kp_k^{\alpha_{\ell}}(x)$ has only real zeros for each $\ell \in \mathbb{N}$.

\medskip

 \noindent The three claims constituting the remainder of the proof are dedicated to showing that 
\[
f_{\alpha_{\ell}}(x) \longrightarrow \sum_{k=0}^nm_k\gamma_k{b_k(x)}
\]
locally uniformly. This fact, together with (i) the fact that each $f_{\alpha_{\ell}}(x)$ has only real zeros and (ii) Hurwitz' theorem, will establish that $\sum_{k=0}^nm_k\gamma_k{b_k(x)}$ has only real zeros, and hence $\{\gamma_k\}_{k=0}^{\infty}$ is a B-multiplier sequence.

\medskip

\noindent \underline{{\bf Claim 1}} \quad As $\ell \to \infty$, $p_k^{\alpha_{\ell}}$ converges locally uniformly to ${b_k(x)}$ for every $k=0,1,2, \ldots, n$.

\noindent {\bf Reason:} \quad Let $K$ be a compact subset of $\mathbb{C}$, let $\epsilon>0$ be given and set $\displaystyle{R=\max_{k,\ x \in K}|b_k(x)|}$. We calculate
\begin{eqnarray*}
|p_k^{\alpha_{\ell}}-{b_k(x)}|&=&\left|\sum_{j=0}^k\dfrac{a_{k,j}{b_j(x)}}{a_{k,k}{\alpha_{\ell}}^{k-j}}-{b_k(x)}\right|\\
&=&\left|{b_k(x)}+\sum_{j=0}^{k-1}\dfrac{a_{k,j}{b_j(x)}}{a_{k,k}{\alpha_{\ell}}^{k-j}}-{b_k(x)}\right|\\
&=&\left|\sum_{j=0}^{k-1}\dfrac{a_{k,j}{b_j(x)}}{a_{k,k}{\alpha_{\ell}}^{k-j}}\right|\\
&\leq&\sum_{j=0}^{k-1}\left|\dfrac{a_{k,j}}{a_{k,k}{\alpha_{\ell}}^{k-j}}\right||{b_j(x)}|\\
&\leq&\sum_{j=0}^{k-1}\left|\dfrac{a_{k,j}}{a_{k,k}{\alpha_{\ell}}^{k-j}}\right|R^j<\epsilon \qquad \forall x \in K, \quad \ell>>1
\end{eqnarray*}
since $\alpha_{\ell} \rightarrow \infty$ as $\ell \rightarrow \infty$. This establishes Claim 1.

\medskip

\noindent \underline{{\bf Claim 2}} \quad $\displaystyle{\lim_{\ell \to \infty} c_{\alpha_{\ell},k}=m_k}$ for all $k=0,1,2,3,\ldots,n$.

\noindent {\bf Reason:} \quad By expanding $f(x)$ in the $Q_{\alpha_{\ell}}^*$ and $B$ bases we get
\[
\sum_{k=0}^nc_{{\alpha_{\ell}},k}p_{\alpha_{\ell}}^k(x)=\sum_{k=0}^nm_k{b_k(x)},
\]
which, after writing the $p_{\alpha_{\ell}}^k(x)$s in terms of the $b_k(x)$s, gives
\[
(\dag) \quad \sum_{k=0}^nc_{{\alpha_{\ell}},k}\left(\sum_{j=0}^k\dfrac{a_{k,j}b_j(x)}{a_{k,k}\alpha_{\ell}^{k-j}} \right)=\sum_{k=0}^nm_k{b_k(x)}.
\]
Note that the coefficients of the $n^{th}$ degree terms on either side of this equation are $c_{\alpha_{\ell},n}$ and $m_n$ respectively. Thus we must in fact have $c_{\alpha_{\ell},n}=m_n$.
Using this, we rewrite $(\dag)$ as 
\[
m_n\left(\sum_{j=0}^n\dfrac{a_{n,j}b_j(x)}{a_{n,n}\alpha_{\ell}^{n-j}} \right)+\sum_{k=0}^{n-1}c_{{\alpha_{\ell}},k}\left(\sum_{j=0}^k\dfrac{a_{k,j}b_j(x)}{a_{k,k}\alpha_{\ell}^{k-j}} \right)=m_nb_n(x)+\sum_{k=0}^{n-1}m_k{b_k(x)}.
\]
We now look at the coefficients of the degree $n-1$ terms, and conclude that
\[
m_n \frac{a_{n,n-1}}{a_{n,n} \alpha_{\ell}}+c_{\alpha_{\ell},n-1}=m_{n-1},
\]
which implies that $\displaystyle{\lim_{\ell \to \infty} c_{\alpha_{\ell},n-1}=m_{n-1}}$. Continuing in this fashion until the process terminates we get that
\[
\lim_{\ell \to \infty} c_{\alpha_{\ell},k}=m_k \qquad  (k=0,1,2,3,\ldots,n)
\]
as desired.

\medskip

\noindent \underline{{\bf Claim 3}} $f_{\alpha_{\ell}}(x)$ converges locally uniformly to $\sum_{k=0}^nm_k\gamma_k{b_k(x)}$ as $\ell \to \infty$.

\noindent {\bf Reason:} \quad Let $K$ be a compact subset of $\mathbb{C}$, and let $\epsilon>0$ be given.
\begin{eqnarray*}
\left|f_{\alpha_{\ell}}(x)-\sum_{k=0}^nm_k\gamma_k{b_k(x)}\right|&=&\left|\sum_{k=0}^nc_{{\alpha_{\ell}},k}\gamma_kp_k^{\alpha_{\ell}}(x)-\sum_{k=0}^nm_k\gamma_k{b_k(x)}\right|\\
&=&\left|\sum_{k=0}^nc_{{\alpha_{\ell}},k}\gamma_kp_k^{\alpha_{\ell}}(x)-m_k\gamma_k{b_k(x)}\right|\\
&\leq&C\sum_{k=0}^n\left|c_{{\alpha_{\ell}},k}p_k^{\alpha_{\ell}}(x)-m_k{b_k(x)}\right|,
\end{eqnarray*}
where $\displaystyle{C=\max_k \{|\gamma_0|, |\gamma_1|, \ldots, |\gamma_n| \}}$. In addition,
\begin{eqnarray*}
\left|c_{{\alpha_{\ell}},k}p_k^{\alpha_{\ell}}(x)-m_k{b_k(x)}\right|&\leq& \left|c_{{\alpha_{\ell}},k}p_k^{\alpha_{\ell}}(x)-c_{{\alpha_{\ell}},k}b_k(x)\right|+\left|c_{{\alpha_{\ell}},k}b_k(x)-m_k{b_k(x)}\right|\\
&=&\left|c_{{\alpha_{\ell}},k}\right| \left|p_k^{\alpha_{\ell}}(x)-b_k(x)\right|+\left|c_{{\alpha_{\ell}},k}-m_k\right| \left|{b_k(x)}\right|\\
&< & \frac{\epsilon}{n} \qquad (\ell \gg 1)
\end{eqnarray*}
as a consequence of Claims 1 \& 2. Thus, for all $x \in K$, and $\ell \gg 1$, we have
\[
\left|f_{\alpha_{\ell}}(x)-\sum_{k=0}^nm_k\gamma_k{b_k(x)}\right| <\epsilon,
\]
which establishes Claim 3, and completes the proof of Theorem \ref{classicalequivalence}.

\end{proof}

\begin{rem} \label{simplesetcontHermite} The converse of Theorem \ref{classicalequivalence} is false in general. Consider the containments 
\begin{eqnarray*} \set{ \textrm{Generalized Laguerre multiplier sequences}} &\subsetneq& \set{\textrm{Hermite multiplier sequences}} \\
\set{\textrm{Legendre multiplier sequences}} &\subsetneq& \set{\textrm{Hermite multiplier sequences}} 
\end{eqnarray*} 
established in \cite{tomandrzej} and in \cite{bdfu} respectively. These containments coupled with the fact that $\seq{r^k}$ is a Hermite multiplier sequence if and only if $r \geq 1$ (Theorem 127 in \cite{andrzej}) provide bases $Q$ and $B$ for which the converse of Theorem \ref{classicalequivalence} fails.
\end{rem}
Nonetheless, if the set $Q$ is the standard basis, the converse of Theorem \ref{classicalequivalence} does hold, as we next demonstrate.

\begin{thm} \label{converse} Let $B=\{b_k(x)\}_{k=0}^\infty$ be a simple set of polynomials. If every classical multiplier sequence is a $B$-multiplier sequence, then there exists an $\alpha>1$ such that $\displaystyle{\left \{\alpha^k\right \}_{k=0}^\infty}$ and $\displaystyle{\left \{\left(1/\alpha\right)^k\right\}_{k=0}^\infty}$ are both $B$-multiplier sequences.
\end{thm}

\begin{proof}

Suppose that every classical multiplier sequence is a $B$-multiplier sequence and suppose that $\displaystyle{ f(x) = \sum_{k=0}^n a_k x^n}$ has only real zeroes.  Since the transformation $x\mapsto\beta x$ preserves reality of zeroes for all non-zero, real $\beta$, we have that for all $\alpha>1$,\quad $f(\alpha x)=\displaystyle{\sum_{k=0}^\infty\alpha^k a_k x^k}$ and $f\left(\frac{x}{\alpha}\right)=\displaystyle{\sum_{k=0}^\infty}\frac{1}{\alpha^k}a_kx^k$ both have only real zeroes. Thus both $\displaystyle{\left \{\alpha^k\right \}_{k=0}^\infty}$ and $\displaystyle{\left \{\left(1/\alpha\right)^k\right\}_{k=0}^\infty}$ are classical multiplier sequences, and therefore $B$-multiplier sequences.

\end{proof}
It is a natural question to ask whether the existence of an $\alpha >1$ as in Theorems \ref{classicalequivalence} \& \ref{converse} implies the existence of more than one such $\alpha$. We settle this question in the next Proposition.

\begin{prop}
Let $B=\{b_k(x)\}_{k=0}^\infty$ be a simple set of polynomials. The following are equivalent:
\begin{itemize}
\item[(i)] There exists an $\alpha >1$ such that $\displaystyle{\left \{\alpha^k\right \}_{k=0}^\infty}$ and $\displaystyle{\left \{\left(1/\alpha\right)^k\right\}_{k=0}^\infty}$ are both $B$-multiplier sequences.
\item[(ii)] $\displaystyle{\left \{\alpha^k\right \}_{k=0}^\infty}$ and $\displaystyle{\left \{\left(1/\alpha\right)^k\right\}_{k=0}^\infty}$ are $B$-multiplier sequences for all $\alpha>1$.
\end{itemize}
\end{prop}
\begin{proof} It is clear that $(ii)$ implies $(i)$. To see why $(i)$ implies $(ii)$, suppose that there exists an $\alpha >1$ such that $\displaystyle{\left \{\alpha^k\right \}_{k=0}^\infty}$ and $\displaystyle{\left \{\left(1/\alpha\right)^k\right\}_{k=0}^\infty}$ are both $B$-multiplier sequences. By Theorem \ref{classicalequivalence} (and the remark following the statement of the theorem), every $Q$-multiplier sequence for every simple set $Q$ is a $B$-multiplier sequence. In particular, every classical multiplier sequence is a $B$-multiplier sequence. But by Theorem \ref{converse}, $\displaystyle{\left \{\alpha^k\right \}_{k=0}^\infty}$ and $\displaystyle{\left \{\left(1/\alpha\right)^k\right\}_{k=0}^\infty}$ are classical multiplier sequences for all $\alpha>1$, and hence they are also $B$-multiplier sequences for all $\alpha >1$.
\end{proof}
 We close this section with a classification theorem, which is an immediate consequence of the results contained in the section.
 
 \begin{thm} \label{classicalclassification} Let $B=\{b_k(x)\}_{k=0}^\infty$ be a simple set of polynomials. The set of $B$-multiplier sequences coincides with the set of classical multiplier sequences if and only if there exists an $\alpha >1$ such that $\displaystyle{\left \{\alpha^k\right \}_{k=0}^\infty}$ and $\displaystyle{\left \{\left(1/\alpha\right)^k\right\}_{k=0}^\infty}$ are both $B$-multiplier sequences.  \end{thm}

\section{Partitions}\label{partitions}

A quick observation reveals that given any $\alpha >1$, both the sequence $\{\alpha^k\}_{k=0}^{\infty}$ and the sequence $\displaystyle{\left\{\left(1/\alpha \right)^k\right \}_{k=0}^{\infty}}$ are classical multiplier sequences. By the remark following Theorem \ref{classicalequivalence} we conclude that if $Q$ is a simple set of polynomials, then every $Q$-multiplier sequence must be a classical multiplier sequence. As a result, $Q$-multiplier sequences inherit a list of properties from the classical multiplier sequences. In particular, if $Q$ is a simple set of polynomials and $\seq{\gamma_k}$ is a $Q$-multiplier sequence, then the following hold:
\begin{itemize}
\item[(i)] if there exists integers $n> m \geq 0$ such that $\gamma_m \neq 0$ and $\gamma_n=0$, then $\gamma_k=0$ for all $k \geq n$,
\item[(ii)] the terms of $\{\gamma_k\}_{k=0}^{\infty}$ are either all of the same sign, or they alternate in sign,
\item[(iii)] for any $r \in \mathbb{R}$, the sequence $\{r \gamma_k\}_{k=0}^{\infty}$ is also a $Q$-multiplier sequence,
\item[(iv)] the terms of $\{\gamma_k\}_{k=0}^{\infty}$ satisfy Tur\'an's inequality
\[
\gamma_k^2-\gamma_{k-1}\gamma_{k+1} \geq 0 \quad \quad (k=1,2,3, \ldots).
\] 
\end{itemize}
Since the set of classical multiplier sequences contains all $Q$-multiplier sequences for every simple set $Q$, we ask whether there exists a collection of simple sets of polynomials $\left\{ Q_j \right \}$, such that the sets of $Q_j$-multiplier sequences forms a partition for the set of all classical multiplier sequences?  One immediately notices that this cannot be true in the traditional sense of a partition, since the intersection of the sets of $Q_j$-multiplier sequences is trivially non-empty: it contains all constant sequences.  In what follows, we establish that this intersection contains very little more than the constant sequences. We show that a logical choice for a `minimal' partition are two sets of generalized Hermte polynomials $\mathcal{H}^{(\alpha)}$: one with with $\alpha >0$ and one with $\alpha <0$, and exhibit that despite its appeal, this choice fails to produce the required partition. 

\medskip

\noindent {\bf Notation:} For ease of exposition given a simple set of polynomials $Q$, we denote by $Q_{MS}$ the set of $Q$-multiplier sequences. In addition we shall write $SSP$ for the set of all simple sets of polynomials.

\medskip

 We begin with a useful result which allows us to decide whether  a classical multiplier sequence is a geometric sequence just  by looking at the first three terms of the sequence.
\begin{prop}\label{geometric}
Suppose $\{\gamma_k\}_{k=0}^\infty$ is a classical multiplier sequence with the first three terms in geometric progression, i.e.,
\begin{equation}\label{geomeq}
\frac{\gamma_1}{\gamma_0} = \frac{\gamma_2}{\gamma_1} = \alpha
 \end{equation}
for some $\alpha \in \mathbb{R}$. Then $\{\gamma_k\}_{k=0}^\infty$ is a geometric sequence, where
\[\gamma_n = \gamma_0\alpha^n\]
for all $n \in \mathbb{N}$.
\end{prop}

\begin{proof}
We proceed by induction. Let $n > 2$, and assume that $\gamma_m = \gamma_0\alpha^m$ for all $m<n$. (We are given that this is true for $n = 3$).  We will show that $\gamma_n = \gamma_0\alpha^n.$\\
From the algebraic characterization of multiplier sequences,
\[\sum_{k=0}^n\gamma_k\binom{n}{k}x^k\]
must have only real zeros. But this is equivalent to saying
\begin{eqnarray*}
\sum_{k=0}^n \gamma_{n-k}\binom{n}{k}x^k &=& \sum_{k=1}^n \gamma_{n-k}\binom{n}{k}x^k + \gamma_n \\
&=& \sum_{k=1}^n \gamma_0\alpha^{n-k}\binom{n}{k}x^k + \gamma_n \\
&=& \gamma_0\left(x+\alpha\right)^n + \left(\gamma_n - \gamma_0\alpha^n\right)
\end{eqnarray*}
has only real zeros. The transformation $(x+\alpha)\mapsto x$ preserves reality of zeros, so
 \[\gamma_0x^n + \left(\gamma_n - \gamma_0\alpha^n\right)\]
 must have only real zeros. But polynomials of the form $ax^n + b$ have only real zeros for $n>2$ iff b = 0. So $\gamma_n - \gamma_0\alpha^n$ must be zero, hence $\gamma_n = \gamma_0\alpha^n$.\\\end{proof}

\begin{rem}
With the choice $\alpha = 1$ in equation \ref{geomeq} we obtain that a classical multiplier sequence $\seq{\gamma_k}$ with $\gamma_0 = \gamma_1 = \gamma_2$ must be a constant sequence. 
\end{rem}
\begin{thm} Let $\displaystyle{S=\bigcap_{Q \in SSP} Q_{MS}}$. If $\{\gamma_k\}_{k=0}^\infty \in S$, then one of the following holds:
\begin{itemize}
\item[(i)] $\{\gamma_k\}_{k=0}^\infty$ is a constant sequence, or
\item[(ii)] $\{\gamma_k\}_{k=0}^\infty=\{ \gamma_0, \gamma_1, 0,0,0,\ldots \}$
\end{itemize}  
\end{thm}
\begin{proof} Suppose $\{\gamma_k\}_{k=0}^\infty \in S$. Since a linear polynomial has only real zeros, it is immediate that sequences of the form $\{ \gamma_0, \gamma_1, 0,0,0,\ldots \}$ are in $S$. Thus we proceed by assuming that $\{\gamma_k\}_{k=0}^\infty \in S$ is not of this type and we show that in that case it must be a constant sequence. To this end consider the following three simple sets of polynomials:
\begin{eqnarray*}
Q_1&=&\left\{ 1,x,x+x^2,x^3,x^4, \ldots, x^n, \ldots \right\} \\
Q_2&=&\left\{1,x+1, x^2+x+1, x^3, x^4, \ldots, x^n,\ldots \right\} \qquad \textrm{and}  \\
Q_3&=&\left\{ 1,x,1+x^2,x^3,x^4, \ldots, x^n, \ldots \right\}. 
\end{eqnarray*}
The polynomial $p(x)=4x^2+4x+1$ has only real zeros. Since $\{\gamma_k\}_{k=0}^\infty$ is a $Q_1$ multiplier sequence, it follows that $\tilde{p}(x)=\gamma_0+4 \gamma_2 x+4 \gamma_2x^2$ has only real zeros as well, and hence we must have
\[
16\gamma_2^2-16\gamma_2\gamma_0=16 \gamma_2 (\gamma_2-\gamma_0) \geq 0.
\]
This implies that $\gamma_2$ and $\gamma_2-\gamma_0$ have the same sign, or $\gamma_2=\gamma_0$. Expanding $f(x)=x^2$ in $Q_3$ and applying $\{\gamma_k\}_{k=0}^\infty$ leads to the polynomial $\gamma_2x^2+(\gamma_0-\gamma_2)$, which should have only real zeros, since $\{\gamma_k\}_{k=0}^\infty$ is also a $Q_3$ multiplier sequence. Thus $\gamma_2$ and $\gamma_0-\gamma_2$ also have the same sign, unless $\gamma_2=\gamma_0$. We conclude that $\gamma_0=\gamma_2$. Note that $\gamma_0=\gamma_2 \neq 0$, since we assumed that our sequence $\{\gamma_k\}_{k=0}^\infty$ is not of the second type. Thus, we may assume that (after perhaps a division by a constant factor) $\gamma_0=\gamma_2=1$. Let 
\[
f(x) = ax^2 + bx + \frac{b^2}{4a}
\] where $a = \gamma_1 + 1$ and $b = \gamma_1 + 2$. Then $f(x)$ has only real zeros, since $b^2 - (4ab^2)/(4a) = 0$.  Expanding $f(x)$ in terms of the basis $Q_2$ we get 
\[
f(x) = a(x^2 + x + 1) + (x + 1) + \left(\frac{b^2}{a} - b\right).
\]
Since $\{\gamma_k\}_{k=0}^\infty$ is also a $Q_2$ multiplier sequence, we see that the polynomial
\begin{eqnarray}
f^*(x) &=& a(x^2 + x + 1) + \gamma_1(x + 1) + \left(\dfrac{b^2}{4a} - b\right)\nonumber\\
&=& ax^2 + (a + \gamma_1)x + a + \gamma_1 + \dfrac{b^2}{4a} - b\nonumber\\
&=& ax^2 + (2\gamma_1 + 1)x + \left(\gamma_1 - 1 + \dfrac{b^2}{4a}\right)\nonumber
\end{eqnarray}
also has only real zeros. Hence the discriminant of $f^*$, given by $\Delta=1-\gamma_1^2$, must be non-negative. Consequently $\gamma_1^2 \leq 1$. On the other hand, by Tur\'an's inequality, $\gamma^2_1 \geq 1$. Thus $\gamma_1^2=1$, from which we conclude that $\gamma_1=1$. (Note that $\gamma_1\neq-1$ since $f(x) = x^2 = (x^2 + x + 1) - (x + 1)$ has only real zeros, but $f^*(x) = (x^2 + x + 1) + (x + 1) = x^2 + 2x + 2 = (x + 1)^2 + 1$ does not have any real zeros.) Finally, by invoking Proposition \ref{geometric} we see that $\{\gamma_k\}_{k=0}^\infty$ is a geometric sequence with $\alpha=1$, and is hence a constant sequence. The proof is complete.


\end{proof}

Next we give a quick generalization of Theorem 5.6 in \cite{tomandrzej} by proving the following, somewhat surprising fact: if $\seq{\gamma_k}$ is a multiplier sequence for a set of simple polynomials with only simple zeros, then it is a $\mathcal{H}^{(\alpha)}$-multiplier sequence\footnote{In \cite{andrzej} Piotrowski shows that the set of Hermite multiplier sequences coincides with the set of $\mathcal{H}^{(\alpha)}$-multiplier sequences for every $\alpha > 0$.} for every $\alpha >0$.  In particular, if $Q$ is a simple set of orthogonal polynomials, then every $Q$-multiplier sequence is a Hermite multiplier sequence (for more orthogonal polynomials see \cite{szego}).

\begin{thm} \label{Hermiteallcontain} Suppose that $Q=\seq{q_k(x)}$ is a simple set polynomials such that each $q_k(x)$ has only simple real zeros, normalized so that leading coefficient of each $q_k(x)$ is positive. If $\seq{\gamma_k}$ is a $Q$-multiplier sequence, then $\seq{\gamma_k}$ is a Hermite multiplier sequence.
\end{thm}

\begin{proof} It suffices to establish the result for non-negative sequences. If the result holds for all non-negative sequences, and $\seq{\gamma_k}$ is any $Q$-multiplier sequence, then either 
\[
\seq{\gamma_k}, \seq{-\gamma_k}, \seq{(-1)^k\gamma_k} \ \textrm{or} \ \seq{(-1)^{k+1}\gamma_k}
\]
 is a non-negative sequence, and is hence a Hermite multiplier sequence. But then $\seq{\gamma_k}$ is itself a Hermite multiplier sequence as a consequence of properties (ii) and (iii) at the beginning of this section. Thus for the rest of the proof we assume that $\seq{\gamma_k}$ is a non-negative Q-multiplier sequence. Observe that the set
 \[
 E_n=\set{b \in \R \ \big| \  q_n(x)+b q_{n-2}(x) \quad \textrm{has only real zeros}} \qquad (n \geq 2)
 \]
 is (i) closed, essentially as consequence of Hurwitz' theorem, and (ii) bounded above by a positive number, since 
 \[
 \frac{d^{n-2}}{dx^{n-2}}\left( q_n(x)+bq_{n-2}(x) \right)=k_1 x^2+ k_2x+bk_3 \qquad (k_1,k_3 >0)
 \]
 has complex zeros for large enough $b$, and hence so does $q_n(x)+bq_{n-2}(x)$. The fact that the upper bound is positive follows from the simplicity of the zeros of $q_n(x)$, for this condition implies that $(-\ve_n,\ve_n) \subset E_n$ for some $\ve_n >0$. If $\seq{\gamma_k}$ is of the form
 \[
 \ldots, 0,0,\gamma_n, \gamma_{n+1},0,0,\ldots
 \]
 for some $n \in \mathbb{N}$, then it is automatically a (trivial) Hermite multiplier sequence. If $\seq{\gamma_k}$ is not of this form, then there must exist an $m\in \mathbb{N}$ such that $\gamma_k=0$ for $k<m$ and $\gamma_k \neq 0$ for $k \geq m$. Using this fact, we show that $\seq{\gamma_k}$ must be non-decreasing. To this end let $b_n=\max E_n$, and consider the polynomial
 \[
 q_n(x)+b_nq_{n-2}(x),
 \]
 which has only real zeros. Since $\seq{\gamma_k}$ is a Q-multiplier sequence, it follows that
 \[
 \gamma_n q_n(x)+\gamma_{n-2}b_n q_{n-2}(x)=\gamma_n \left(q_n(x)+\frac{\gamma_{n-2}}{\gamma_n}b_n q_{n-2}(x) \right)
  \]
  has only real zeros as well.  By the maximality of $b_n$ we must then have $\displaystyle{0<\frac{\gamma_{n-2}}{\gamma_n} \leq 1 }$. On the other hand, using Tur\'an's inequality we see that 
  \[
  \left(\frac{\gamma_{n-1}}{\gamma_{n-2}} \right)^2 \geq \frac{\gamma_{n}}{\gamma_{n-2}} \geq 1, 
  \]
  which in turn implies that $\gamma_{n-1} \geq \gamma_{n-2}$ for $n \geq m+2$. Since the same inequality holds trivially for $n <m+2$, we conclude that $\seq{\gamma_k}$ in non-decreasing. Finally, by Proposition 151 in \cite{andrzej}, every non-decreasing non-negative classical multiplier sequence is a Hermite multiplier sequence. The proof is complete.
 \end{proof}

The preceding theorem together with Proposition 151 in \cite{andrzej} suggest that when trying to partition the set of all classical multiplier sequences, one should look to the generalized Hermite bases. In line with this suggestion, at the end of his dissertation, Piotrowski posed the following problem (Problem 165, p. 152): 
\begin{quotation} If $\seq{\gamma_k}$ is a classical multiplier sequence, then does there exists a non-zero real constant $\alpha$ such that $\seq{\gamma_k}$ is an $\mathcal{H}^{(\alpha)}$-multiplier sequence? 
\end{quotation}
We are now in position to answer this question in the negative. 

\begin{lem} \label{noninc} Let $\seq{\gamma_k}$ be a classical multiplier sequence of non-negative terms. If there exists $n \in \mathbb{N}$ such that $\gamma_n \leq \gamma_{n-1}$. Then $\gamma_k \leq \gamma_{k-1}$ for all $k \geq n$.
\end{lem}
\begin{proof} Suppose that $n\in \mathbb{N}$ is such that $\gamma_n \leq \gamma_{n-1}$. Then
\[
\gamma_n(\gamma_n-\gamma_{n+1}) \geq \gamma_n^2-\gamma_{n-1}\gamma_{n+1} \stackrel{*}{\geq} 0,
\]
where the starred inequality is Tur\'an's inequality. If $\gamma_n=0$ then $\gamma_k=0$ for all $k\geq n$. Otherwise we conclude that $\gamma_{n+1} \leq \gamma_n$. The result follows. 
\end{proof}
\begin{prop} Let $\alpha <0$. If a sequence of non-negative terms $\seq{\gamma_k}$ is a $\mathcal{H}^{(\alpha)}$-multiplier sequence, then $\gamma_{k+1} \leq \gamma_k$ for $k \geq1$.
\end{prop}
\begin{proof} If $\gamma_2=0$, the claim follows immediately. Suppose now that $\gamma_2 \neq 0$ and let $\Gamma=\seq{\gamma_k}$ be a sequence as in the statement of the proposition. Consider the polynomial
\[
x^2=\mathcal{H}^{(\alpha)}_2+\alpha \mathcal{H}^{(\alpha)}_0.
\]
Then 
\[
\Gamma[x^2]=\gamma_2x^2+\alpha(\gamma_0-\gamma_2)
\]
has only real zeros if and only if $\gamma_0 \geq \gamma_2$. Note that we must have either $\gamma_1\geq \gamma_0\geq \gamma_2$ or $\gamma_0\geq \gamma_1 \geq \gamma_2$. In both cases $\gamma_1\geq \gamma_2$, and hence by Lemma \ref{noninc} $\gamma_{k+1} \leq \gamma_k$ for $k\geq 1$.
\end{proof}

\begin{lem} \label{nowhere}The sequence $\Gamma=\displaystyle{\left\{\frac{1}{8}, 1, 2, 0 ,0,\ldots \right\}}$ is a classical multiplier sequence.
\end{lem}
\begin{proof} By the classification theorem for classical multi[lier sequences in \cite{PS}, it is enough to show that $\Gamma[(1+x)^n]$ has only real, non-positive zeros for $n \geq 1$. The result is immediate if $n=1$. For $n \geq 2$ we have that 
\[
\Gamma[(1+x)^n]=\frac{1}{8}+n x+n(n-1)x^2
\]
is a quadratic polynomial with roots
\[
r_{1,2}=\frac{-\sqrt{2n^2} \pm \sqrt{n(n+1)}}{\sqrt{2}2n (n-1)},
\]
both of which are negative. The proof is complete.
\end{proof}

\begin{prop} There does not exists a non-zero constant $\alpha$ such that the sequence $\displaystyle{\left\{\frac{1}{8}, 1, 2, 0 ,0,\ldots \right\}}$ is a $\mathcal{H}^{(\alpha)}$-multiplier sequence.
\end{prop}
\begin{proof} If $\alpha >0$, every $\mathcal{H}^{(\alpha)}$-multiplier sequence must be non-decreasing. If $\alpha<0$, then every $\mathcal{H}^{(\alpha)}$-multiplier sequence must be non-increasing after the first term. The result follows.
\end{proof}

\section{Open questions}\label{concl}
 We have obtained some first results regarding the partitioning of the set of classical multiplier sequences. Answers to the following questions would greatly enhance our understanding of this problem. \\
 1. \ Given $\alpha>0, \beta<0$, the set $\mathcal{H}^{(\alpha)}_{MS} \cup \mathcal{H}^{(\beta)}_{MS}$ is a strict subset of the set of classical multiplier sequences. There are classical multiplier sequences which `peak' (in the sense of Lemma \ref{nowhere}) in the $n^{th}$ term and are unaccounted for. Is there a (finite) collection of bases which would account for all such multiplier sequences?
 
 \smallskip
 \noindent 2. Is there a collection of infinitely many simple sets of polynomials $\set{Q_j}$such that
 \begin{itemize}
 \item[(i)] $Q_{jMS} \vartriangle Q_{iMS} \neq \emptyset$ for $j\neq i$
 \item[(ii)] $\displaystyle{\bigcup_{j=1}^{\infty} Q_{jMS} \subsetneq Q_{MS}}$
 \end{itemize}
 We know that the collection $\displaystyle{Q_j:=\left \{ 1,1+x,1+x+x^2, \ldots, \sum_{i=0}^j x^i, x^{j+1}, x^{j+2}, \ldots \right\}}$ satisfies $(ii)$, but we do not know whether it satisfies $(i)$. 
 
 \smallskip
 \noindent 3. Are there classical multiplier sequences which are not multi[lier sequences for any other simple set of polynomials? If there are, can one find `maximal' subset (with corresponding simple sets of polynomials) of the set of classical multiplier sequences?


\begin{thebibliography}{99}

\bibitem[BDFU]{bdfu} K. Blakeman, E. Davis, T. Forg\'acs and K. Urabe {\emph{On Legendre Multiplier sequences}}, submitted. Available at {\it arxiv.org/abs/1108.4662}

\bibitem[BC-01]{BC} D. Bleecker and G. Csordas {\emph{Hermite expansions and the distribution of zeros of entire functions}}, Acta Sci. Math. (Szeged), \textbf{67} (2001), 177-196.

\bibitem[BB-09]{BB} J. Borcea and P. Br\"and\'en, {\emph{P\'olya-Schur master theorems for circular domains and their boundaries.}, Annals of Math., \textbf{170}, (2009), 465-492.}

\bibitem[CC-04]{CCsurvey} T. Craven and G. Csordas, {\emph{Composition theorems, multiplier sequences, and complex zero decreasing sequences}, in Value Distribution Theory and Related Topics, Advances in Complex Analysis and Its Applications, Vol. 3, eds. G. Barsegian, I. Laine and C. C. Yang, kluwer Press, 2004}

\bibitem[FP-11]{tomandrzej} T. Forg\'acs, and A. Piotrowski, {\emph{Multiplier Sequences for generalized Laguerre bases}, forthcoming in the Rocky Mountain Journal of Math.}

\bibitem[L-64]{Levin} B. Ja. Levin, {\emph{Distribution of Zeros of Entire Functions}}, Transl. Math. Mono. Vol. 5, Amer. Math. Soc., Providence, RI, 1964; revised ed. 1980.

\bibitem[M-49]{Marden} Marden, M., {\it The geometry of the zeros of a polynomial in a complex variable}, Mathematical Surveys Number III, American Mathematical Society, 1949.

\bibitem[O-63]{O} N. Obreschkoff, {\emph{Verteilung und Berechnung der Nullstellen Reeller Polynome}}, Veb Deutscher Verlag der Wissenschaften, Berlin, 1963.

\bibitem[P-07]{andrzej} Piotrowski, A., {\it Linear Operators and the Distribution of Zeros of Entire Functions}, Ph. D. dissertation

\bibitem[PS-14]{PS} G. P\'olya and J. Schur, \"Uber zwei Arten von Faktorenfolgen in der Theorie der algebraischen Gleichungen, J. Reine Angew. Math. \textbf{144}, (1914), 89-113.

\bibitem[R-60]{Rainville} Rainville, E.D., {\it Special Functions}, The Macmillan Company, New York, 1960.

\bibitem[Sz-39]{szego} G. Szeg\"o, {\it Orthogonal Polynomials}, American Mathematical Society Colloquium Publications, Vol. XXIII. 1939.

\bibitem[T-50]{Turan} P. Tur\'an, \emph{Sur l'alg\`ebre fonctionnelle}, Compt. Rend. du prem. (Ao\^ut-2 Septembre 1950) Cong. des Math. Hongr., \textbf{27}, Akad\'emiai Kiad\'o, Budapest, (1952), 279-290.

\end{thebibliography}
\end{document}